\documentclass[12pt]{article}
\usepackage[letterpaper,top=2cm,bottom=2cm,left=3cm,right=3cm,marginparwidth=1.75cm]{geometry}

\usepackage{amsmath,amssymb,amsfonts,amsthm}

\usepackage{graphicx}
\newtheorem{theorem}{Theorem}
\newtheorem{remark}{Remark}

\newtheorem{definition}{Definition}
\usepackage{verbatim}
\usepackage{mathtools}

\usepackage{xcolor}
\usepackage{url}

\usepackage[textsize=tiny]{todonotes}
\usepackage{algorithm}
\usepackage{algpseudocode}
\usepackage{mathtools}
\usepackage{xcolor}
\usepackage{natbib}
\usepackage[colorlinks=true, allcolors=blue]{hyperref} %
\usepackage{doi}
\definecolor{darkred}{rgb}{.7,0,0}

\definecolor{darkgreen}{rgb}{.15,.55,0}

\definecolor{darkblue}{rgb}{0,0,0.7}

\newcommand{\email}[1]{\href{mailto:#1}{#1}}

\author{%
 Nisha Chandramoorthy\thanks{\email{nishac@gatech.edu}, Georgia Institute of Technology} \and Florian Sch\"afer\thanks{\email{florian.schaefer@cc.gatech.edu}, Georgia Institute of Technology} \and Youssef Marzouk\thanks{\email{ymarz@mit.edu}, Massachusetts Institute of Technology}}
\title{Score Operator Newton Transport}

\begin{document}
\maketitle
%

%

\begin{abstract}
We propose a new approach for sampling and Bayesian computation 
that uses the score of the target distribution to construct a transport from a given reference distribution to the target. Our approach is an infinite-dimensional Newton method, involving an elliptic PDE, for finding a zero of a ``score-residual'' operator. We prove sufficient conditions for convergence to a valid transport map. Our Newton iterates can be computed by exploiting fast solvers for elliptic PDEs, resulting in new algorithms for Bayesian inference and other sampling tasks. We identify elementary settings where score operator Newton transport achieves fast convergence while avoiding mode collapse.

\end{abstract}

\section{{Introduction}}
\label{sec:intro}
Generating samples from a complex (e.g., non-Gaussian, high-dimensional) probability distribution is a core computational challenge in diverse applications, ranging from computational statistics and machine learning to molecular simulation. A recurring setting is where the density $\rho$ of the target distribution is specified up to a normalizing constant---for example, in Bayesian modeling, where $\rho$ represents the posterior density. Here, evaluations of the \emph{score} $\nabla \log \rho$ are often available as well, even for complex statistical models \citep{villa2021hippylib}. Alternatively, many new methods enable effective score estimation from data, without explicit density estimation; examples include score estimation from time series observations in chaotic dynamical systems \citep{chandramoorthy, angxiu} and score-based modeling of image distributions \citep{song2020score, Song2020DenoisingDI}. 

In these settings, transport or ``flow''-driven algorithms for generating samples have seen extensive success. The central idea is to construct a transport map from a simple, prescribed source distribution to the target distribution of interest. One class of transport approaches, e.g., as represented by variational inference with normalizing flows, involves constructing a \textit{parametric} class of invertible maps and minimizing some statistical divergence between the pushforward (see Section~\ref{sec:preliminaries}) 
of the source by a member of this class and the target. A different, essentially nonparametric, class of transport approaches are based on
particle systems, e.g., Stein variational gradient descent (SVGD) \citep{liu2016stein} and its many variants \citep{li2020stochastic, chen2020projected,detommaso2018stein}. These methods can be interpreted as gradient flows \citep{jko} of some functional on the space of probability measures, for different choices of geometry \citep{chewi2020svgd,duncan2019geometry}. Such methods yield implicit representations of transport maps, through the paths taken by sample trajectories \citep{han2017stein}. Yet another class of transport methods involve \textit{prescribing} continuous paths \citep{masrani2021q,albergo2023stochastic} between the source and target distributions, 
and approximating these paths with particle systems or learned velocity fields. 

Of course, none of these approaches is without drawbacks. Parametric representations of transport maps often involve ad hoc choices of parametric class, where bias must be balanced against complexity of the representation; moreover, the optimization problems that must be solved over such classes seldom have guarantees. On the other hand, gradient flow approaches, as exemplified by SVGD or more generally standard Langevin dynamics, may be slow to converge and quite sensitive to the geometry and dimensionality of the target distribution. 
Because the transport map is not explicitly available in these approaches, it is generally difficult to update the map, e.g., for perturbed scores, or to reuse the map for downstream computations. 
Continuous-time ``homotopy'' approaches require \textit{a priori} selection of a path through the space of probability measures, and may involve solving equations that depend explicitly on estimates of the density at the current time \citep{reich2011dynamical} or otherwise resorting to cruder approximations \citep{iglesias2021adaptive}.

This paper introduces a new sampling approach based on different ingredients: an infinite-dimensional score matching principle and discrete-time dynamics. Specifically, we construct a transport map as the zero of a \emph{score residual} operator via an infinite-dimensional \emph{Newton method} for root finding, which is typically called the Newton--Raphson method in finite dimensions. The transport is a \textit{composition} of maps found, at each step, via solution of a linear elliptic partial differential equation (PDE). We harness regularity theory for elliptic PDEs to prove existence of such a map and to prove convergence of the iterations. The resulting \emph{score-operator Newton} (SCONE) transport construction 
is illustrated in Figure~\ref{fig:figureKAM}. It applies to any sampling problem where the scores of the source and target measures can be evaluated. Several desirable features of our approach are as follows: 
\begin{itemize}

\item Newton methods are \emph{efficient}: we will show, empirically and in simple analytical examples (see Appendix \ref{appx:gaussian}), that very few iterations may be required. The Newton construction also permits an existing map to be updated or fine-tuned, e.g., for perturbed scores; this is useful for applications such as Bayesian filtering.

\item Unlike the nonlinear Monge-Amp\`ere equation, which describes optimal transport maps, our construction involves a sequence of \emph{linear} PDEs, which are more amenable to analysis, fast computation, and dimension reduction.
     
\item Elliptic differential operators instantaneously propagate information throughout the domain. Hence, our transport updates, which use elliptic PDE solutions, are intrinsically \emph{global} \citep{evans2022partial}. As evidenced by our numerical results, our construction thus tends to avoid mode collapse, since transport updates are influenced by score values over the entire support of the distributions, including the tails.
  
\end{itemize}

We summarize our main contributions as follows: We define a score transformation operator that maps an input score and a transport map to the transported score. We prove the existence of transport maps that are fixed points of an operator based on the score-transformation operator and the target score. Our existence proof is constructive and leads to a Newton method on Banach spaces. Our construction yields a transport map and defines a new notion of score-matching in infinite dimensions. Convergence of transport maps, and scores, is established in classical H\"older norms, unlike in dual norms typically used for variational inference methods.

In this paper, we establish the theoretical foundations of score operator Newton transport and provide proof-of-concept numerics. 
Our construction and theory are developed in the infinite-dimensional setting, enabling flexible representations of the transport map. Hence, the construction facilitates the development of many new sampling algorithms, based on kernel methods, deep neural networks, and other discretizations of the underlying linear elliptic PDEs. Developing such scalable algorithms for the Newton updates will be a subject of our future work. 

\begin{figure}
    \centering
    \includegraphics[width=0.8\textwidth]{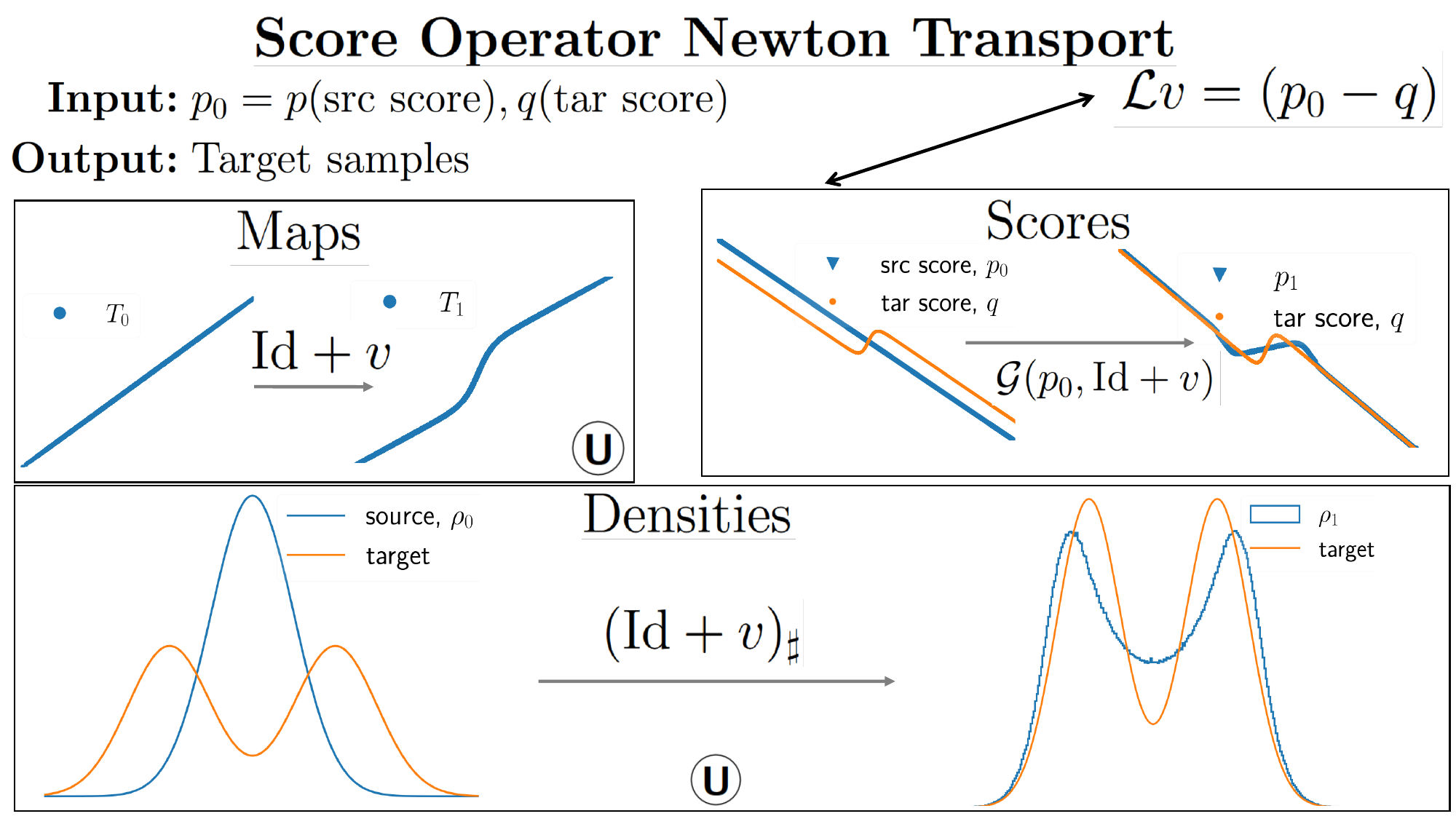}
    \caption{A graphical overview of the construction of transport maps. The \emph{score} of the source and target densities are given as inputs. The method outputs samples from the target distribution. Each iteration involves solving an elliptic PDE that gradually transports the score of the source to that of the target. The PDE solutions across iterations are combined via a simple composition operation to obtain the transport map from the source to the target.}
    \label{fig:figureKAM}
\end{figure}

\section{{Infinite-dimensional score matching}}
\label{sec:preliminaries}
Suppose $\nu$ is our unknown target probability measure on $\mathbb{R}^d$ with associated density $\rho^\nu.$ We define the \emph{score} of the target to be the vector-valued function $q:=\nabla \log \rho^\nu:\mathbb{R}^d\to\mathbb{R}^d.$ In our setting, the target score $q(x)$ is available at every $x \in \mathbb{R}^d.$ 

Let $\mu$ be a source or reference probability measure on $\mathbb{R}^d$ with density $\rho^\mu$. The source is chosen to be easy to sample from, e.g., a Gaussian in $\mathbb{R}^d$. Its vector-valued score function is defined as $p:= \nabla \log \rho^\mu.$ 
Here $\nabla$ is the gradient operator on Euclidean space. The main contribution of this work is a new transport map to sample from the target $\nu$ using the source samples from $\mu$ and the source and target scores, 
$$ p(x) := \nabla\log \rho^\mu(x), \quad q(x) := \nabla\log \rho^\nu(x).$$ Our transport map is defined as the solution of an infinite-dimensional root finding problem for the score operator. Next we define both these objects: transport maps and the score operator.

\textbf{Transport maps:} Given two probability measures $\mu$ and $\nu$ on $\mathbb{R}^d$, we say that a measurable function $T:\mathbb{R}^d \to \mathbb{R}^d$ is a \emph{transport map} from $\mu$ to $\nu$ if 
\begin{align}
\label{eq:pushforward}
    \nu(A) = (T_\sharp \mu)(A) \coloneqq \mu (T^{-1}(A)),
\end{align}
for every measurable set $A$; here $\sharp$ denotes the pushforward operation. In high-dimensional inference problems, the target $\nu$ that we wish to sample from is a often a complicated ``intractable'' measure. Our goal is to find an invertible map $T$ that transforms samples of $\mu$ to samples distributed according to $\nu$. Recall that samples from the source $\mu$ are easily obtained. 

In general, the transport map $T$ between the source and target measures is not unique. The optimal transport map is one useful and canonical choice, but in the context of Bayesian inference---where our main goal is to simulate $\nu$---it suffices to find any transport $T$ that is computationally feasible. Now we provide a constructive approximation of a map $T$ that satisfies \eqref{eq:pushforward} by exploiting the scores, $p = \nabla \log \rho^\mu$ and $q = \nabla \log \rho^\nu$, associated with the source $\mu$ and target $\nu$, respectively. To do this, we define a \emph{score operator} on the product space of functions representing scores and transport maps.

\textbf{Score operator:} We define a \emph{score operator}, denoted by $\mathcal{G}$, that takes as arguments a source score and an invertible transport map and returns the score of the resulting pushforward distribution. That is, if $T$ is an invertible map such that $T_\sharp \mu = \nu$, the operator $\mathcal{G}$ is defined such that
\begin{align}
    \label{eq:scoreOperator}
    \mathcal{G}(p, T) = q. 
\end{align}
Recall the change of variables formula for probability densities,
\begin{align}
    \rho^\nu = \dfrac{\rho^\mu \circ T^{-1}}{|{\rm det} \nabla T| \circ T^{-1}}.
\end{align}
Taking logarithms and differentiating the above formula, we obtain a definition for the score operator $\mathcal{G}$,
\begin{align}
\label{eq:scoreOperatorDefinition1}
\mathcal{G}(s, U) &= \left(s  (\nabla U)^{-1} - \nabla \log |{\rm det} \nabla U| (\nabla U)^{-1} \right) \circ U^{-1} \\
\label{eq:scoreOperatorDefinition2}
&= \Big(s \, (\nabla U)^{-1} - {\rm tr}\big(
(\nabla U)^{-1} \\
\notag
&\nabla^2 U \big) (\nabla U)^{-1} \Big) \circ U^{-1},
\end{align}
where \eqref{eq:scoreOperatorDefinition2} follows from using Jacobi's formula for the derivative of the determinant. We use ${\rm tr}((\nabla U)^{-1} \nabla^2 U)$ for the vector-valued function $[{\rm tr}((\nabla U)^{-1} \partial_1 \nabla U), \ldots, {\rm tr}((\nabla U)^{-1} \partial_d \nabla U)]^{\top}$, where $\partial_i$ is the partial derivative in the $i$th coordinate.
The above operator takes as arguments a score $s$ associated with a probability measure, say $\pi$, and a $\mathcal{C}^2$-diffeomorphism $U$, to return the score associated with the measure $U_\sharp \pi$.
That is, $\mathcal{G}$ expresses the change of variables, or the pushforward operation, on the space of scores. 
We end this section by listing some properties of $\mathcal{G}$ that will be useful in the sections to follow and can be checked by using the above definitions.  
\begin{enumerate}
    \item[(i)]  $\mathcal{G}(s, {\rm Id}) = s$ for any score $s$, expressing the fact that the identity coupling results in the same probability measure;
    \item[(ii)] $\mathcal{G}(s, {\rm Id} + c) = s\circ ({\rm Id} - c)$, where $c$ is a constant function;
    \item[(iii)] $\mathcal{G}$ has the group property, mimicking the pushforward operator. That is, 
    \begin{align*}
        \mathcal{G}(s, \psi_1 \circ \psi_2) = \mathcal{G}(\mathcal{G}(s, \psi_2), \psi_1).  
    \end{align*}
\end{enumerate}
Note that the operator $\mathcal{G}$ is not injective in the second argument, and hence not invertible. That is, $\mathcal{G}(s, U_1) = \mathcal{G}(s, U_2)$ does not imply that $U_1 = U_2$. Any solution $T$ to \eqref{eq:scoreOperator} is a valid transport map from $\mu$ to $\nu$. We refer to the problem of finding a $T$ that satisfies \eqref{eq:scoreOperator} as the infinite-dimensional score-matching problem. This is because the solution $T$ is a function, and the score of the pushforward distribution through $T$ matches the target score $q.$ Despite its name, our problem is not derived as an infinite-dimensional version of the score-matching problem from \citep{hyvarinen2005estimation}, which is essentially a density estimation problem. 
Here our objective is to obtain a transport map given target scores and to use it for sampling. In the next section, we will define a score-residual operator that maps a score and a transport map to the difference between the target score and the score of the pushforward of the source by the transport map. We will then derive a solution to the score-matching problem as a zero of this operator; thus, our solution strategy for transport also deviates from the variational problem involving the typical score-matching objective in the literature \citep{hyvarinen2005estimation, song2019generative, song2020score, wibisonoSGM}.

\begin{remark}[Availability of scores] As mentioned in the introduction, the target score is available in settings beyond Bayesian inference with a tractable likelihood and prior. In other words, in many settings, scores can be approximated well even though an explicit model for the unnormalized density is not available. One such example is the score of an ergodic, invariant measure of certain classes of chaotic systems. Here, fast methods have been developed to evaluate scores at any point on a chaotic orbit \citep{chandramoorthy,angxiu}.
\end{remark}
\begin{remark}[Validity of transport maps] Note that any diffeomorphism that satisfies $\mathcal{G}(p,T) = q$ is a valid transport map in the sense that $T_\sharp \mu = \nu$. By integrating both sides of \eqref{eq:scoreOperatorDefinition2}, we obtain that any $T$ that satisfies $\mathcal{G}(p,T) = q$ also satisfies, $\rho^\nu = k (\rho^\mu/ |{\rm det} \nabla T|)\circ T^{-1}$, for an integration constant $k.$ Since the left hand side is a valid density (integrates to 1), we obtain that $k$ must be 1.
\end{remark}

\section{{Learning a zero of the score-residual operator}}
\label{sec:construction}
 Fixing $p$ and $q$, we define \emph{the score-residual operator} on the space of $\mathcal{C}^2$-diffeomorphisms on $\mathbb{R}^d$ as  $$\mathcal{R}(T) \coloneqq \mathcal{G}(p,T) - q.$$ A zero of this operator is a transport map between the measures associated with $p$ and $q$. Here we describe an iterative approach for finding a zero of this operator, which is a generalization of Newton's method to infinite-dimensional spaces. 
\subsection{Score operator Newton (SCONE) method}
\label{sec:KAMNewton}
Infinite-dimensional generalizations of the Newton method appear in the analysis of PDEs under the name of Nash--Moser iteration \citep{berti2015nash}. They also appear in the context of finding \emph{conjugacies} between nearby dynamical systems, in an approach called the Kolmogorov--Arnold--Moser or KAM method \citep{moser1961new}.
In the next section, we derive sufficient conditions for the convergence of this method.

To develop the SCONE iteration, we expand the score operator, supposing that $p$ is close to $q$ and assuming a linear structure on function space near $(q,{\rm Id})$, 
\begin{align}
\label{eq:expansionOfG}
	\mathcal{G}(p, T) &= \mathcal{G}(q, {\rm Id}) + D_1\mathcal{G}(q, {\rm Id}) (p - q) + \\
\notag
	&D_2 \mathcal{G}(q, {\rm Id}) (T - {\rm Id}) + \Delta(p, T),
\end{align}
where ${\rm Id}$ is the identity function on $\mathbb{R}^d$, $D_1, D_2$ are first-order partial derivatives (Frech\'et derivatives) of $\mathcal{G}$ in its first and second arguments respectively, and $\Delta(p,T)$ contains nonlinear operators on $p-q$ and $T-{\rm Id}$. Analogous to the elementary Newton method, to find the solution to the score-matching problem, we first look for a solution to the \emph{linearized} score-matching problem. That is, we find such a $T$ for which the left hand side of \eqref{eq:expansionOfG} is $q$ and $\Delta(p,T) = 0$. This linearization of the score-matching problem yields,
\begin{align}
    - D_1\mathcal{G}(q,{\rm Id})\: (p-q) = D_2\mathcal{G}(q,{\rm Id}) v,
\end{align}
defining the vector field $$v \coloneqq T - {\rm Id}.$$ Since $\mathcal{G}(p,{\rm Id}) = p$ for all $p$, we have that $D_1\mathcal{G}(q, {\rm Id}) = {\rm Id}$. We can explicitly compute the linear operator $D_2\mathcal{G}(q, {\rm Id})$ to be the following differential operator, which, for convenience, we define as $\mathcal{L}(q)$, 
\begin{align}
    \label{eq:definitionOfL}
    -D_2 \mathcal{G}(q, {\rm Id})\: v = \mathcal{L}(q) \: v := \nabla q\: v + q \:\nabla v + {\rm tr}(\nabla^2 v).
\end{align}
In the above, the trace ${\rm tr}(\nabla^2 v)$ of the tensor $\nabla^2 v$ is a row vector with the $i$th column being ${\rm tr}(\partial_i \nabla v)$.
Using this, we obtain that $v$ satisfies,
\begin{align}
    \label{eq:newtonUpdateDerivation}
    (p - q) = \mathcal{L}(q)\: v.
\end{align}
We may then iterate this update in the following way. Assuming that $\mathcal{L}(q)$ is invertible, we obtain the solution $v = (\mathcal{L}(q))^{-1} (p-q)$. Using $T = {\rm Id} + v$, we obtain $p_1 = \mathcal{G}(p,T)$, and then, we repeat the update \eqref{eq:newtonUpdateDerivation} with $p_1$ replacing $p$, and solve for $v_1$. We then update the transport map approximation as $T_1 = ({\rm Id} + v_1) \circ ({\rm Id} + v)$. Proceeding further, at the $n$th iteration, we obtain the function $v_n$ by solving
\begin{align}
\label{eq:newtonStep}
     - (q - p_n) = \mathcal{L}(q)v_n =  (\nabla q) v_n + q_n (\nabla v_n) + {\rm tr}(\nabla^2 v_n).
\end{align}
Then, we set,
\begin{align}
\notag 
    T_{n+1} &\leftarrow ({\rm Id} + v_n) \circ T_n \\
    \label{eq:newtonUpdate}
    p_{n+1} &\leftarrow \mathcal{G}(p_n, {\rm Id} +v_n).
\end{align}
In the same spirit as the KAM method from the dynamical systems literature, notice that the differential operator $\mathcal{L}(q)$ remains the same across steps, \eqref{eq:newtonStep}. In the next section, we will give sufficient conditions under which the sequence of functions $(T_n)_{n \geq 0}$ converges, in a classical H\"older norm, to a transport map $T$, i.e., an invertible map that satisfies $T_\sharp \mu = \nu$. 
\begin{algorithm}
\caption{Score-operator Newton Transport}\label{alg:KAMNewton}
\begin{algorithmic}
\State $T_0(x) = x,$ $p_0(x) = \nabla \log \rho^\mu(x),$ $x_1, x_2,\cdots, x_m \sim \mu$
\While{$n\leq n_{\rm max}$}
\State $ v_n \gets \mathcal{L}(q)^{-1}(p_n - q)$
    \State $x_i \gets x_i + v_n(x_i)$
    \State $p_{n+1}(x) \gets \mathcal{G}(p_n, {\rm Id}+v_n)(x)$
    \State $n\to n+1$ 
\EndWhile
\State Return $\{x_i\}_{1\leq i\leq m}$
\end{algorithmic}
\end{algorithm}
\subsection{SCONE transport algorithm}
\label{sec:scON}
The steps of the infinite-dimensional Newton algorithm derived above are summarized in Algorithm \ref{alg:KAMNewton}. The input to the algorithm are black-box functions that return the source score $p_0 = p$ and $q$, the target score. In addition, we have $m$ iid samples from the source distribution, say, $\{x_i\}_{i=1}^m.$ After $n$ steps of the algorithm, these samples are transformed to $\{T_n(x_i)\}_{i=1}^m,$ which represent target samples more accurately as $n$ increases. The algorithm can also return the transport map $T$ as a black-box function that can be evaluated at any point.

In the beginning, we set the initial guess for the transport map $T_0$ (e.g., $T_0(x) = {\rm Id}(x) = x$) and $p_0 = p$, the source score. At each iteration, we solve the PDE \eqref{eq:newtonStep} to obtain the vector field $v_n$. There is extensive literature on approximating solutions of PDEs using neural networks, such as physics-informed neural networks \citep{raissi2019physics}, deep Ritz methods \citep{lu2021priori,yu2018deep} and Fourier neural operators \citep{li2020fourier}; as well as kernel methods and Gaussian process approximations \citep{owhadi2019operator,wendland2004scattered,schaback2006kernel,zhang2000meshless}, and finite elements \citep{ciarlet2002finite,wang2013weak}. Bespoke methods that exploit the particular structure of $\mathcal{L}(q)$ are deferred to future work (see Section \ref{sec:discussion}). 

Using meshfree methods such as PINNs or FNOs, we obtain a black-box solution $v_n$ that can be evaluated at any point. Moreover, these black-box solutions, being neural networks, can also be automatically differentiated. Thus, we can evaluate $v_n(x_i), \nabla v_n(x_i)$ and $\nabla^2 v_n(x_i)$, where $\{x_i\}_{i=1}^m$ are the samples at the $n$th iteration. When we use grid-based methods, we obtain (approximate) evaluations of $v_n$ at the grid points. Then, we use interpolations and finite-differences to obtain $v_n(x_i)$ and its derivatives. Using \eqref{eq:scoreOperatorDefinition2}, we update the source score $p_n$ as $p_{n+1} = \mathcal{G}(p_n, {\rm Id} + v_n).$ The samples are also transformed as $x_i \to x_i + v_n(x_i)$. Note that this results, at the end of the $n$th iteration, in the samples $\{T_n(x_i)\},$ where $T_n := ({\rm Id} + v_n) \circ T_{n-1}.$ We repeat this process until the algorithm converges, i.e., $\|v_n\|$ is close to zero or if the maximum number of iterations is reached.

\textbf{SCONE complexity:} The computational cost of our method is ${\cal O}(n\:(C_{\rm pde} + C_{\rm u})) $, where $n$ is the number of Newton iterations, $C_{\rm pde}$ is the cost of the solving the PDE and $C_{\rm u}$ is the cost of the update step \eqref{eq:newtonUpdate}. Typically Newton methods exhibit quadratic convergence \citep{GALANTAI200025}, which can be accelerated further for finite-dimensional problems under local smoothness conditions \citep{gerlach1994accelerated} (see Section \ref{sec:discussion}).
The cost of solving the PDE dominates the per iteration cost and na\"ive methods typically have a computational complexity that scales exponentially with the problem dimension ($d$). However, in the context of elliptic PDEs, various sparse structures in the solution have been exploited to mitigate the curse of dimensionality, including tensor decompositions \citep{tensorDahmen}, hierarchical low-rank approximations \citep{boulle2023learning}, and other notions of model reduction (see Section \ref{sec:discussion}).   

\subsection{Related work}
\label{sec:comparison}
In many variational methods for sampling and Bayesian inference, one seeks transport maps that minimize a divergence or distance functional on a space of probability measures, over a parametric class of maps $\mathcal{U}$. As an example, one could define a score-based distance metric, and seek a $T$ such that
\begin{align}
\label{eq:aapoScore}
    T = \arg \, \min_{U \in \mathcal{U}} \int \| q(x) - \mathcal{G}(p, U)(x) \|^2  \: d\nu(x).
\end{align}
Candidate maps $U \in \mathcal{U}$ are parameterized and an empirical minimization problem for the parameters is solved using a method that is appropriate for the distance functional. Common classes of parametric maps include normalizing flows \citep{papamakarios2021normalizing}, the flows of neural ODEs \citep{chen2018neural}, and gradients of input-convex neural networks \citep{huang2020convex}.
%
An alternative class of approaches produces \emph{implicit} representations of transport maps through paths taken by sample trajectories of a deterministic or stochastic dynamics. Beginning with the classical work of \citet{jko}, these dynamics are derived such that their mean field limits are gradient flows of the distance functional on the space of probability measures, for appropriate choices of objective and geometry \citep{chewi2020svgd,duncan2019geometry,han2017stein}. Methods that fall into this class include SVGD and Langevin dynamics. (See \citet{wibisono2018sampling} for an overview of the connection between sampling methods and optimization of distance functionals on probability spaces.)

Variational autoencoders \citep{kingma2019introduction} or GANs \citep{goodfellow2020generative} also transport a low-dimensional source measure to a target, learning parameterized decoders or generators by minimizing a variety of objectives. These methods have no obvious dynamical interpretation; their stable training and obtaining theoretical guarantees are challenging.

The SCONE method exploits scores but does not make use of parametric spaces to define transport maps. Rather than minimizing a distance functional, we derive a generalized Newton--Raphson method on Banach spaces to construct a zero of the operator $\mathcal{R}(T)$. This approach specifies the transport map as a composition of functions that is achieved in the limit of a  discrete-time dynamical system---as opposed to a continuous-time flow---on function spaces. Our SCONE transport defines, correspondingly, a discrete-time dynamical system on the space of probability densities, with the target being a fixed point. 

\section{{Convergence proof}}
\label{sec:proofs}
Here we prove the convergence of the SCONE algorithm in \ref{alg:KAMNewton}, establishing a new construction of a transport map. We do not compute explicit bounds on the error norms, $\|p_n - q\|$. We invoke elliptic regularity theory to establish convergence. Our Newton iterates \eqref{eq:newtonStep} yield second-order linear, elliptic PDEs, as we describe below. Let $M$ be a compact subset of $\mathbb{R}^d$ and $\Omega$ be an open set containing $M$.
At each step, we solve $d$ second-order PDEs of the following form,
\begin{align}
\label{eq:systemOfPDEs}
    \left(L(x,D) v\right)_i = f_i,\;1\leq i\leq d,
\end{align}
where the linear differential operator $L$ is given by the $d\times d$ matrix with $$L(x, D)_{ij} = \sum_{|\alpha| \leq 2} a_\alpha^{ij}(x) \: D^\alpha.$$
Here, $\alpha = (\alpha_1, \cdots, \alpha_d)$ is a multi-index, $D^\alpha = \partial^{\alpha_{1}}_1\cdots \partial^{\alpha_{d}}_d$, such that $|\alpha| = \alpha_1+\cdots +\alpha_d \leq 2$. Suppose we parameterize the solutions $v_n$ of the system as the gradient $\nabla \phi_n$ of some differentiable function $\phi_n:\Omega \to \mathbb{R}.$ Then, substituting into \eqref{eq:newtonStep}, we obtain the following equation for the scalar function $\phi_n,$
\begin{align*}
    \nabla(\nabla\cdot \nabla \phi_n) + \nabla (\nabla\phi_n\cdot q) = p_n - q.
\end{align*}
Integrating this equation, we find that our Newton iterates, when $v_n = \nabla \phi_n$ satisfies,
\begin{align}
\label{eq:sconeScalar}
    \nabla\cdot \nabla \phi_n + \nabla \phi_n \cdot q = \log(\rho^\mu_n/\rho^\nu) + C,
\end{align}
where $C$ is an integration constant and $\rho^\mu_n$ is the density of the pushforward measure, $T_{n\sharp}\mu.$ The operator, $\mathcal{P} = \nabla.\nabla + q\cdot\nabla,$ is an \emph{elliptic} operator, which ensures the well-posedness of solutions to \eqref{eq:sconeScalar}. We recall the ellipticity condition of an operator, which says that the highest-order (principal) symbol is coercive (see e.g., the textbook of \citet{Hormander1963} or \citet{semyon} for notes and \citet{gilbarg1977elliptic} for Schauder estimates). 
\begin{definition}[H\"older space of order $k$ and exponent $\gamma$] The H\"older space of order $k$ and exponent $\gamma$, denoted $\mathcal{C}^{k,\gamma}(\Omega)$, is a Banach space consisting of functions that have continuous derivatives up to order $k$ and $\gamma$-H\"older continuous $k$th order derivatives. It is complete with respect to the following norm, for $k \geq 1$, $f:\Omega \to \mathbb
{R}$,
\begin{align}
    \|f\|_{k,\gamma} := \|f\|_k + \max_{|\alpha| = k} \sup_{x\in \Omega} \|D^\alpha f\|_{0,\gamma}, 
\end{align}
where,
\begin{align}
    \|f\|_k &:= \max_{|\alpha| \leq k} \sup_{x \in \Omega} |D^\alpha f(x)| \\
    \|f\|_{0,\gamma} &:= \|f\|_0 + \sup_{x,y \in \Omega, x\neq y} \dfrac{|f(x) - f(y)|}{|x-y|^\gamma}.
\end{align}
We define H\"older continuous cotangent vector fields of the form $v = \nabla \phi \in D^*\bar{\Omega}$ for some differentiable function $\phi,$ where $D^*\bar{\Omega}$ is used to denote the cotangent bundle (rather than the usual $T^*\bar{\Omega},$ to avoid confusion with transport maps $T$).  If $\phi \in \mathcal{C}^{k, \gamma}(\Omega)$, then, we have that the components of $v$ (interpreted as scalar functions) are $\mathcal{C}^{k-1, \gamma}(\Omega).$ Note that, due to the (compact) embedding of H\"older spaces, $\mathcal{C}^{0,\gamma} \to \mathcal{C}^{0,\delta}$ for $\delta < \gamma$, $f \in (\mathcal{C}^{k,\gamma}(\Omega))^d$ with components $f_i \in \mathcal{C}^{0, \gamma_i}(\Omega)$ implies that $\gamma \leq \min_i \gamma_i$. In a slight abuse of notation, for a vector-valued function $f \in (\mathcal{C}^{k,\gamma}(\Omega))^d,$ including cotangent vector fields, we write,
\begin{align}
\label{eq:vectorNorm}
    \|f\|^2_{k,\gamma} := \sum_{i=1}^d \|f_i\|^2_{k,\gamma}.
\end{align}
\end{definition}

We use Schauder estimates of the type below \citep{gilbarg1977elliptic} from classical elliptic PDE theory.
\begin{theorem}\label{thm:elliptic}
        Let $\Omega$ be a bounded and open subset of $\mathbb{R}^d$ with a smooth boundary. Let $L(x,D) u = f$ be a second-order strongly elliptic system, with $L(x,D) = \sum_{|\alpha| \leq 2}  a_\alpha(x) D^\alpha$, and zero Dirichlet boundary conditions. If the coefficients $a_\alpha$ and the right hand side $f$ are in $\mathcal{C}^{s,\gamma}(\bar{\Omega})$, then, $u \in \mathcal{C}^{s+2,\gamma}(\bar{\Omega})$. In particular, for any $s \geq 0$, and $\gamma \in (0,1),$
    $$\|u\|_{s+2,\gamma} \leq K (\|f\|_{s,\gamma} + \|u\|_s),$$
where $K$ only depends on $\|a_\alpha\|_{s,\gamma}$ and $d$.
\end{theorem}

\begin{theorem}[Score-matching]
\label{thm:score-matching}
Let $\Omega$ be a bounded, open subset of $\mathbb{R}^d$ containing the origin, with a smooth boundary. Let $q \in \mathcal{C}^{s+1,\cdot}(\bar{\Omega})$ be the score of a target density $\rho^\nu \in \mathcal{C}^{s+2,\cdot}(\bar{\Omega})$. Then, for every $\epsilon > 0, s \in \mathbb{N}$ there exists a $\delta > 0$ such that for any reference density $\rho^\mu$ with associated score $p$ such that $\|p - q\|_s \leq \epsilon$, there is a transformation $T \in \mathcal{C}^{s+2,\cdot}(M)$ such that (i) $\mathcal{G}(p,T) = q$ and (ii) $\|T - {\rm Id}\|_{s+2} \leq \delta$.
\end{theorem}
\begin{proof}
Define
\begin{align}
	\notag
	\mathcal{H}(v) &:= \mathcal{G}(q, {\rm Id} + v) = \Bigg( q  \: ({\rm Id} + \nabla v)^{-1} - \\ 
	\notag
	&{\rm tr}(({\rm Id} + \nabla v)^{-1}\: \nabla^2 v) \\ 
	\label{eq:defH}
	&({\rm Id} + \nabla v)^{-1} \Bigg)\circ ({\rm Id} + v)^{-1}.
\end{align}
From definition \eqref{eq:defH}, note that $\mathcal{H}(0) = q$. We can explicitly compute the first derivative of $\mathcal{H}$ at 0 to be
\begin{align}
    d\mathcal{H}(0) w = - \nabla q\: w - q \: \nabla w - {\rm tr}(\nabla^2 w). 
\end{align}
 We can also deduce that $d\mathcal{H}(0) \nabla \phi$ is equivalent to $-\nabla \mathcal{P} \phi,$ where $\mathcal{P} = \nabla \cdot \nabla + q\cdot \nabla$ is an elliptic operator, which is Fredholm on $\mathcal{C}^{s+2,\gamma}(\bar{\Omega})$. Let $C^{s,\gamma}(\bar{\Omega})$ denote the quotient space of $\mathcal{C}^{s,\gamma} (\bar{\Omega})$ corresponding to  the equivalence relation $f \sim g$ if $\nabla f(x)= \nabla g(x),$ at all $x \in \bar{\Omega}.$ Correspondingly, we define a closed subspace of $\mathcal{C}^s$ cotangent vector fields, $\mathcal{V}^{s,\gamma} := \{ x\to v(x) = \nabla \phi(x) \in D_x^* \bar{\Omega}, \phi \in {C}^{s+1,\gamma} (\bar{\Omega}), x \in \Omega \},$ with norm $\|\nabla \phi\|_* = \|\nabla \phi\|_{s,\gamma}$ and let $B_\theta^{s+2,\gamma}(0)$ be a $\theta$-ball around 0 in $\mathcal{V}^{s+2,\gamma}$. The element, $\phi,$ in a sufficiently small set around the constant element in $C^{s+3,\gamma}(\bar{\Omega})$ can identify an element, $\nabla \phi,$ of $B_\theta^{s+2,\gamma}(0).$  We note that the operator $\mathcal{H}:B^{s+2,\gamma}_\theta(0) \to (\mathcal{C}^{s,\gamma}(\bar{\Omega}))^d$ is well-defined as a continuous operator. 

When $d \mathcal{H}(0)$ is defined on $\mathcal{V}^{s+2,\gamma},$ and using Theorem \ref{thm:elliptic}, we know that its kernel only contains the zero element of $\mathcal{V}^{s+2,\gamma},$ which corresponds to $\mathcal{P}\phi =$ const $\in C^{s,\gamma}(\bar{\Omega})$. Thus, we obtain that $d\mathcal{H}(0):\mathcal{V}^{s+2,\gamma}\to \mathcal{V}^{s, \gamma}$ is bijective. In particular, for a fixed $q$, the $s+2$-H\"older norm of $w$ that solves $d\mathcal{H}(0) w = f$ for an $f \in \mathcal{V}^{s,\gamma}$ is bounded above, by Theorem \ref{thm:elliptic}. Hence, $d\mathcal{H}(0)^{-1}$ is continuous on $\mathcal{V}^{s,\gamma}$. 

Thus, we can apply the inverse function theorem for $\mathcal{H}$. There exists an open neighborhood, $B^{s,\gamma}(q,\epsilon)$, of radius, say $\epsilon$, of $q$ in $\mathcal{V}^{s, \gamma}$ and a continuously differentiable map $\mathcal{I}:B^{s,\gamma}(q,\epsilon)\to \mathcal{V}^{s+2,\gamma}$ so that $\mathcal{I} \circ \mathcal{H}(v) = v$. Thus, for any $p = \mathcal{H}(v) \in B_{s,\gamma}(q,\epsilon)$, the map $T = ({\rm Id} + v)^{-1}$ is such that 
$\mathcal{G}(p, T) = q$. This proves $(i)$. From the continuity of $\mathcal{I}$ at $q$ in $B_{s,\gamma}(q,\epsilon)$, we can choose a $\delta_0 > 0$ such that $\|(T - {\rm Id})^{-1}\|_{s+2,\gamma} \leq \delta_0$. This implies that for some $\delta > 0$, $\|v\|_{s+2,\gamma}\leq \delta$, hence proving (ii).  
\end{proof}
The above is an existence result for a transport map and is established via the inverse function theorem for the score operator. It is important to note that even though the result is local (that is, for nearby probability densities), uniqueness cannot still be established for the transport map. This is because, in the above proof, the operator $d\mathcal{H}(0)$ has a non-empty kernel on function spaces containing functions of the form, $\nabla \phi,$ for some $\phi \in \mathcal{C}^{s+2,\gamma}(\bar{\Omega}).$ For any function $\phi$ such that $\mathcal{P}\phi = f + {\rm const},$ $v = \nabla \phi$  solves $d\mathcal{H}(0) v = \nabla f,$ and therefore $d\mathcal{H}(0)$ is not injective. In other words, we have not shown that there is only one transport map between a given source and target density, even when they are close to each other. 
We have defined quotient spaces on which to define $d\mathcal{H}(0)$ to make it invertible, using isomorphism theorems for vector spaces. 

Proving the inverse function theorem via the Banach fixed point theorem both establishes the existence of the desired map $T$ and also the means to construct $T$ as a fixed point iteration of the contraction map. In the theorem below, we explicitly define such a contraction map whose fixed point is $T$. Further, the fixed point iteration of the map is equivalent to our SCONE iteration. 
Such an interpretation of the Newton--Raphson method as a fixed point iteration of the linearization of the given map is indeed classical in numerical analysis, when we are interested in finding zeros of a function on a finite-dimensional space. The following theorem extends this idea to infinite-dimensional spaces and serves as the convergence proof of the SCONE method. 

\begin{theorem}[SCONE construction of transport] \label{thm:kam-newton} When $q \in (\mathcal{C}^{s,\gamma}(\bar{\Omega}))^d$, there exists a $\theta > 0$ such that for any $p$ in a $\theta$-neighborhood of $q$, $p_n \to q$ in $(s,\gamma)$-H\"older norm, where,  
\begin{align}
\notag
    v_n &= \mathcal{L}(q)^{-1} (p_n-q) \\
\label{eq:iterationInsideTheorem}
    p_{n+1} &= \mathcal{G}(p_n, {\rm Id} + v_n),\;\;
    n\in \mathbb{Z}^+, \: p_0 = p.
\end{align}
\end{theorem}
\begin{proof}
Recall the definition of $\mathcal{L}(q)$ from \eqref{eq:definitionOfL}. Note that $\mathcal{L}(q)$ is not invertible on $(\mathcal{C}^{s,\gamma}(\bar{\Omega}))^d,$ and $v_n$ in the statement of the theorem refers to any cotangent vector field $v_n = \nabla \phi_n$ such that $\mathcal{L}(q) v_n = p_n - q.$ We show in the proof of Theorem \ref{thm:score-matching} that $\mathcal{L}(q)^{-1}$ is a homeomorphism between a $\theta$-neighborhood of $q$ in $\mathcal{V}^{s,\gamma}$ (see the proof of \ref{thm:score-matching} for the definition of this space) and a $\mathcal{V}^{s+2,\gamma}$ neighborhood of zero. Thus, choosing $\theta > 0$ sufficiently small, we can combine both the SCONE iteration and update steps (\eqref{eq:newtonStep} and \eqref{eq:newtonUpdate}) to define the operator,
\begin{align}
    \mathcal{J}(v) = \mathcal{L}^{-1}(\mathcal{G}(q + \mathcal{L}v, {\rm Id} + v) - q),
\end{align}
where we write $\mathcal{L}$ to indicate $\mathcal{L}(q)$ for a fixed target score $q$. It is clear that $\mathcal{J}(0) = 0$ and hence $0$ is a fixed point of $\mathcal{J}$. The operator $\mathcal{J}$ is smooth at $0$. In particular, through direct computation, we can verify that its first derivative, $d\mathcal{J}(0) = 0$ as an operator from $B^{s+2,\gamma}_\theta(0)$ to $\mathcal{V}^{s+2,\gamma}$. Applying the mean value theorem, we get,
\begin{align}
\notag
    \|\mathcal{J}(v_1) - \mathcal{J}(v_2)\| &\leq \|v_1 - v_2\| 
    \\
    &\sup_{\eta \in (0,1)} \|d\mathcal{J}(\eta v_1 + (1-\eta) v_2)\|.  
\end{align}
Since $d\mathcal{J}(0) = 0$ and $d\mathcal{J}$ is continuous, we can choose a $\delta < \theta$ such that for all $v \in B_\delta^{s+2,\gamma}(0)$, $\|d\mathcal{J}(v) \| \leq (1/2)$. Thus, we obtain that $\mathcal{J}$ is a contraction on $B_\delta^{s+2,\gamma}(0)$. Thus, the fixed point iteration of $\mathcal{J}$, i.e., $v_{n+1} = \mathcal{J}(v_n)$, starting from any $v_0 \in B_\delta^{s+2,\gamma}(0)$ converges to 0. Note that $\|v_0\| = \|\mathcal{L}^{-1}(p_0 - q)\| \leq C (p_0 - q)$, from the continuity of $\mathcal{L}^{-1}$. Thus, if $p_0 \in B_\epsilon(q)$, one can choose $\delta_0 := \min\{\delta, C\epsilon\}$ such that $\mathcal{J}$ is a contraction on $B_{\delta_0}^{s+2,\gamma}(0)$ and hence, the iteration converges. 

Now we show that the convergence of $v_n \to 0$ implies the convergence of $T_n$ to a transport $T$. Recall that 
$T_k = \circ_{n=0}^k ({\rm Id} + v_n)$, with the compositions being on the left. Hence, $\|T_n - T_{n-1}\| = \| v_n \circ T_{n-1}\| \leq \|v_n\| \to 0$. Finally, to see that the limit $T := \lim_{n\to \infty} T_n$ is a transport, from \eqref{eq:iterationInsideTheorem} for a finite $n$, $p_{n+1} = \mathcal{G}(p_{n}, {\rm Id} + v_n) = \mathcal{G}(p_0, T_n)$, by applying the group property of $\mathcal{G}$ iteratively.
Taking the limit $n \to \infty$ on both sides, we obtain $q = \mathcal{G}(p_0, T)$.
\end{proof}
\begin{remark}
	Instead of a Newton method, one can also define a different fixed-point iteration on an operator defined using the score-residual operator (see Appendix \ref{appx:fixedPoint}). However, under similar smoothness assumptions, these fixed point iteration methods typically show slower convergence \citep{smale1985efficiency}.
\end{remark}
\begin{figure}
	\includegraphics[width=0.3\textwidth]{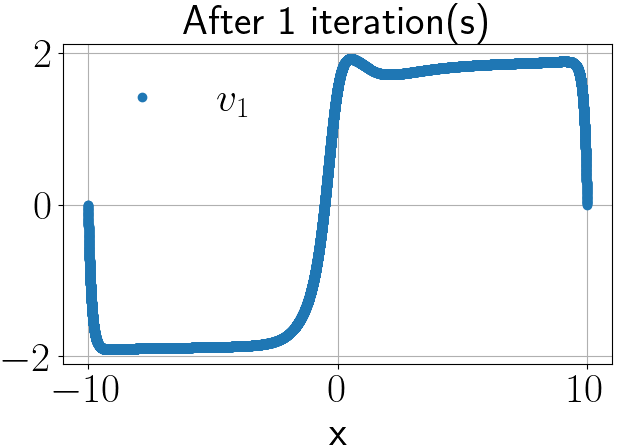}
	\includegraphics[width=0.3\textwidth]{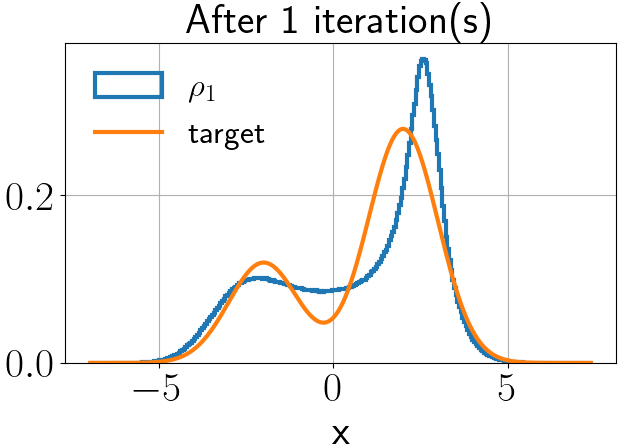}
	\includegraphics[width=0.3\textwidth]{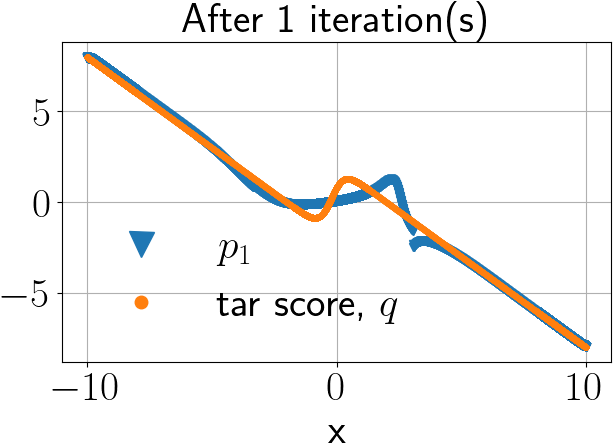}\\
	\includegraphics[width=0.3\textwidth]{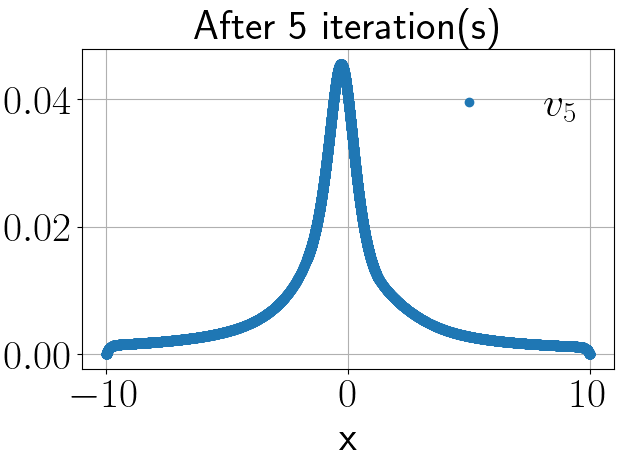}
	\includegraphics[width=0.3\textwidth]{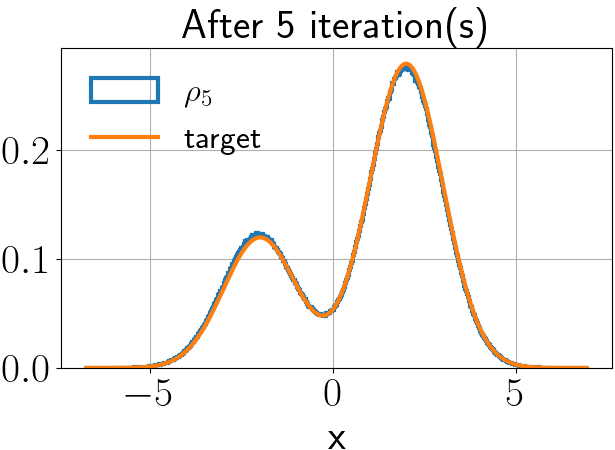}
	\includegraphics[width=0.3\textwidth]{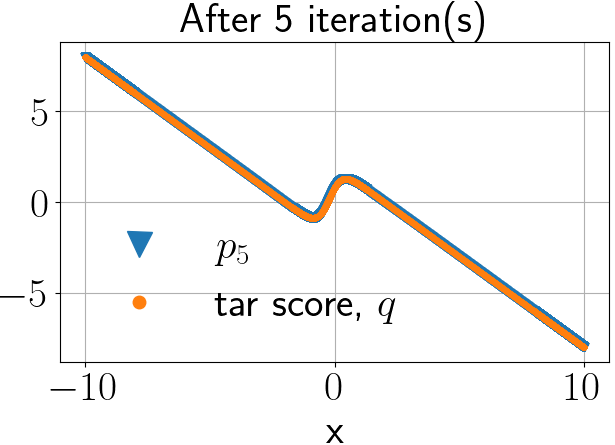}
	\caption{Left: the solution $v_n$ after 1 (top) and 5 (bottom) iterations computed using second order finite difference with {\cal O}(500) grid points. Center: the transformed empirical density. Right: $p_1$ (top) and $p_5$ (bottom).}
 \label{fig:bimodal0.3}
\end{figure}

\section{{Numerical results}}
\label{sec:numerics}
Here we present proof-of-concept numerical results that demonstrate the convergence of our SCONE transport (Algorithm \ref{alg:KAMNewton}). 
On 1D domains, we find that the SCONE transport map converges to the monotone map or the increasing rearrangement, which is optimal with respect to any convex cost (see Chapter 2 of \citet{santambrogio2015optimal}). We also demonstrate that SCONE transport maps can effectively tackle multimodality in the target. See Appendix \ref{sec:appx-num} for details on the numerical methods and additional experiments.
\begin{figure}
	\includegraphics[width=0.34\textwidth]{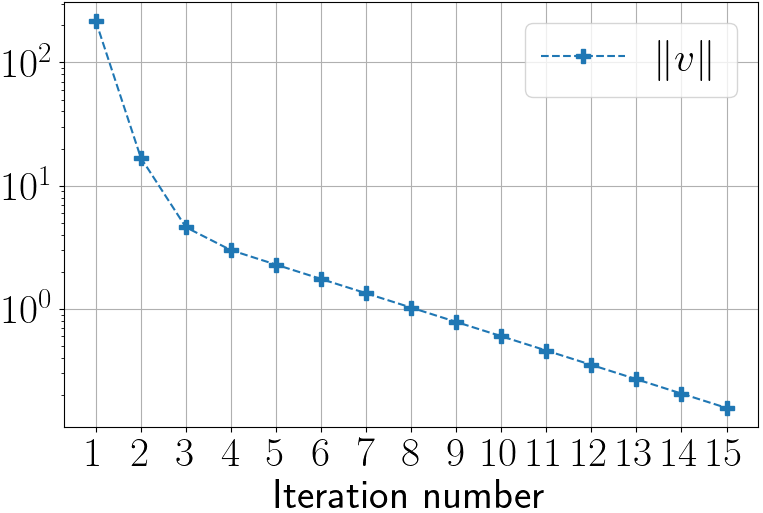}
	\includegraphics[width=0.3\textwidth]{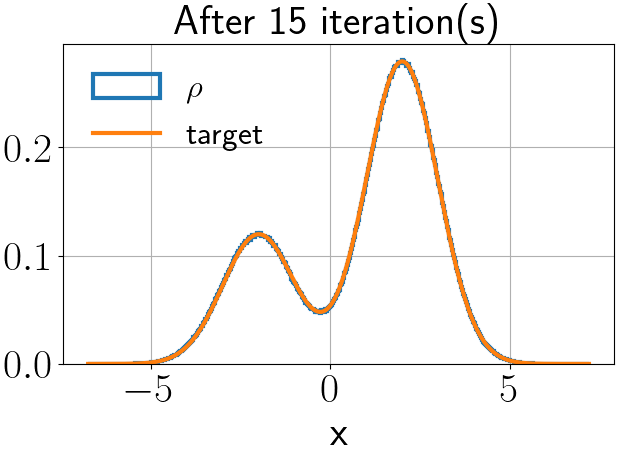}
	\includegraphics[width=0.3\textwidth]{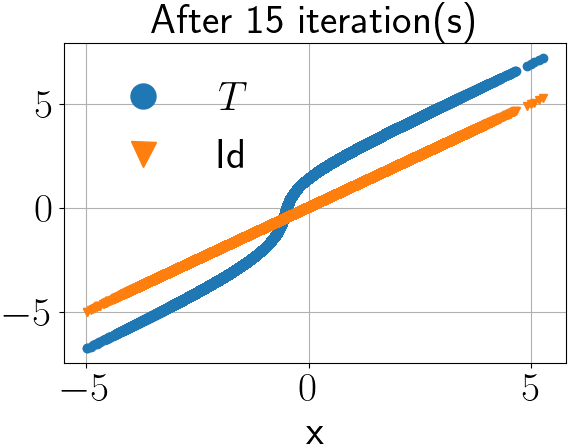}
	\caption{Convergence of SCONE. Left: the convergence of $\|v_n\|$. Center: transformed empirical density after 15 iterations. Right: $T_n$ after 15 iterations.}
 \label{fig:bimodal0.3_conv}
\end{figure}

In Figure \ref{fig:bimodal0.3}, we show the results of applying the SCONE algorithm to a bimodal target density of the form $w_1 \mathcal{N}(m_1, \sigma_1^2) + w_2 \mathcal{N}(m_2, \sigma_2^2)$ (shown in orange in the center column). The target score is shown in orange on the right column. We see from the second row of Figure \ref{fig:bimodal0.3} that the transformed scores and densities match those of the target closely after just 5 iterations. We observe numerically that solving the SCONE step \eqref{eq:newtonStep} (which reduces to an ODE in 1D) on a coarser grid (of size ${\cal O}(100)$) does not affect fast convergence when we add a small $\ell^2$ regularization. In comparison, SVGD takes ${\cal O}(500)$ iterations \citep{liu2016stein} with 100 particles for the same problem, and a vanilla GAN implementation \citep{goodfellow2020generative} can lead to unstable training and mode collapse \citep{thanh2020catastrophic,li2018limitations}. In Figure \ref{fig:bimodal0.3_conv}, we show the convergence of the SCONE algorithm to the target. We see that the identified transport map converges to the optimal map (shown in blue in the right column).

\begin{figure}
    \centering
    \includegraphics[width=0.3\textwidth]{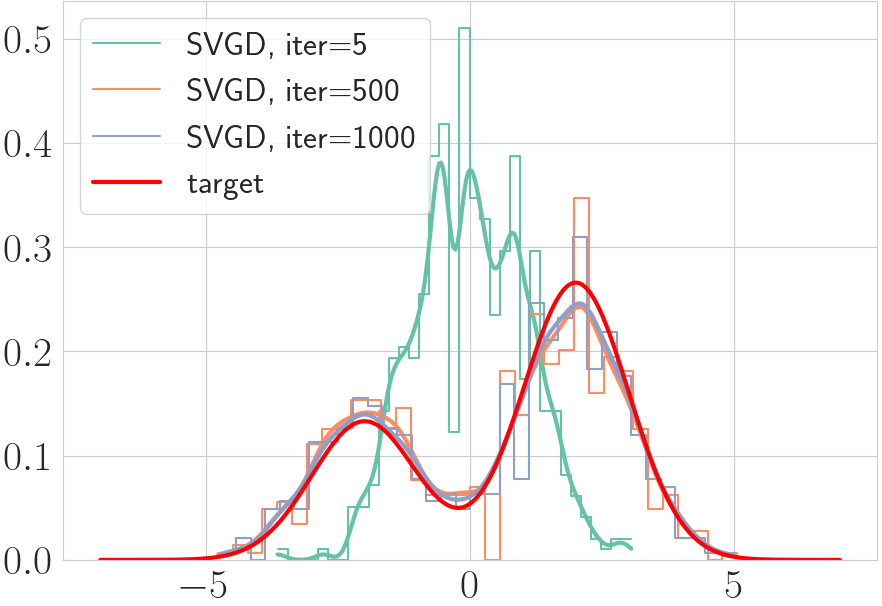}
    \includegraphics[width=0.3\textwidth]{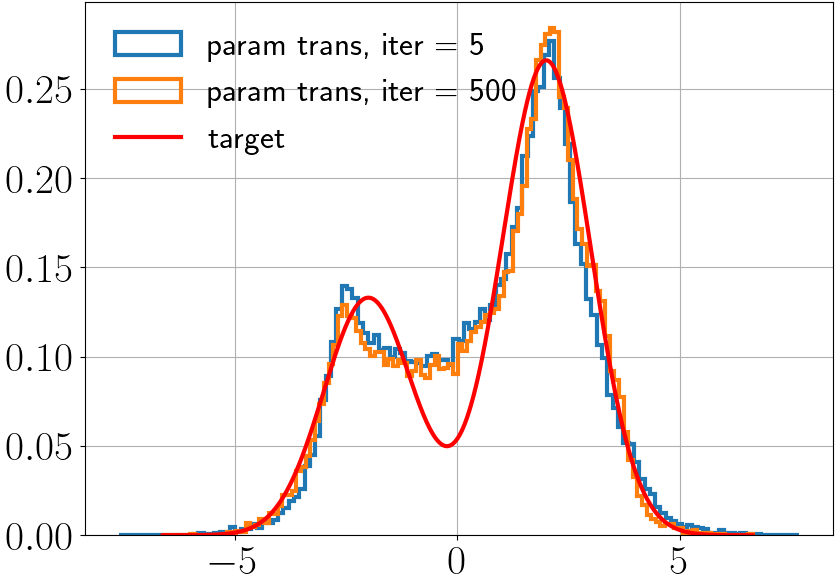}
    \includegraphics[width=0.3\textwidth]{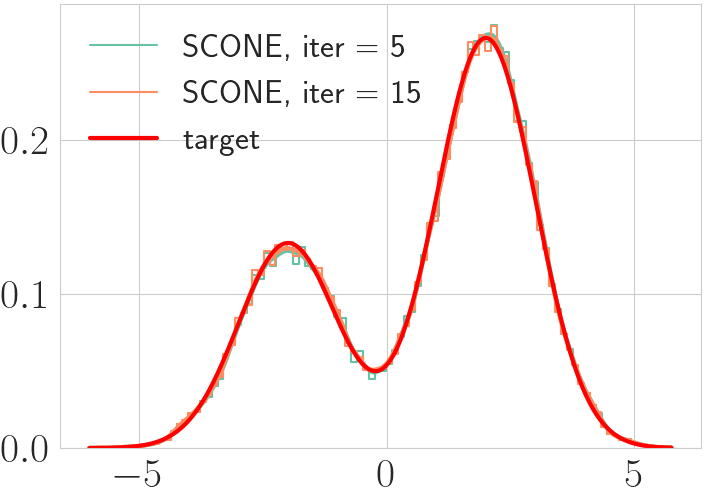}
    \caption{\textit{Left}: SVGD with RBF kernels (median heuristic for bandwidth) and 512 particles. 
    \textit{Center}: parameterized monotone transport map \citep{parno2022mpart}, with polynomial degree 10, 512$^2$/11 samples, optimized using gradient descent + line search. \textit{Right}: SCONE transport with ODE updates solved with 512 grid points.}
    \label{fig:comparison}
\end{figure}

\textbf{Comparison with other algorithms} In Figure \ref{fig:comparison}, we compare SCONE against SVGD \citep{liu2016stein} and parameterized transport maps \citep{parno2022mpart} as these are the most widely used classes of algorithms for Bayesian inference. With $N$ samples, an SVGD iteration has ${\cal O}(N^2)$ cost. With $G$ grid points, SCONE iterations have ${\cal O}(G^2)$ cost, since the linear system to be solved at a SCONE iteration is banded.
With $P$ parameters and $S$ samples to evaluate the variational (KL) objective, the optimization step for a parameterized transport map is ${\cal O}(PS).$ We choose $PS = N^2 = G^2 = 512^2$ so that the computational budget \textit{per iteration} remains the same across the different methods. We see in Figure \ref{fig:comparison} (right) that densities obtained from SCONE most closely match the target (red) after just 5 iterations, while SVGD (left) takes ${\cal O}(500)$ iterations to converge to a comparatively worse solution. 
Parameterized transport (center) converges quickly, with the help of a backtracking line search optimization, but makes a significant approximation error which persists even for higher polynomial degree (we tested up to 20). Figure \ref{fig:comparison} thus demonstrates that SCONE vastly outperforms SVGD and parameterized transport algorithms for the same computational cost.

\section{{Discussion}}
\label{sec:discussion}
We introduce a new notion of infinite-dimensional score-matching that yields a Newton-type method for sampling, and we prove sufficient conditions for its convergence. Our method applies in settings where scores of the source and target measures are easily computed. We comment on theoretical and algorithmic features of this work that will spur further research.

\textbf{Learning elliptic PDEs:} Many structure-exploiting, fast, and sample-efficient methods are emerging for learning the solution operators of linear elliptic PDEs; see, e.g., \citet{lu2021machine, boulle2023learning,schafer2021sparse}. These methods use randomized numerical linear algebra \citep{boulle2023learning}, CNN-based encoder-decoder networks \citep{ZHANG2023116214}, interpolation between deep neural network and Monte Carlo approximations \citep{nusken2023interpolating}, etc.  
These results suggest that it is possible to develop optimal methods, in terms of computational and sample complexity, for learning our SCONE update operator by exploiting 
low-rank 
structure in its solutions. 
With the same goal, we will also explore particle methods (e.g., from fluid dynamics \citep{monaghan2012smoothed, cottet2000vortex}), which can also use fast numerical linear algebra and have theoretical guarantees. 

\textbf{Newton convergence:} Under typical conditions, the classical result of Kantorovich (see \citet{GALANTAI200025} for a survey) establishes quadratic convergence starting in a ball of sufficiently small radius (as in our Theorem \ref{thm:kam-newton}). In future work, we will investigate damping and modified SCONE iterations to prevent divergence \citep{smale1985efficiency}.  
We will also develop inexact and quasi-Newton variants of SCONE, as a way of further reducing computational cost \citep{traub1980convergence} and allowing for errors in $q.$ We will derive theoretical convergence guarantees for these modified SCONEs.

\vspace{0.5in}
\textbf{Funding} NC and YMM acknowledge support from NSF grant PHY-2028125. 
YMM acknowledges support from DOE ASCR award DE-SC0023187, AFOSR award FA9550-20-1-0397, and ONR award N00014-20-1-2595.
FS acknowledges support from AFOSR award FA9550-23-1-0668 and ONR award N00014-23-1-2545.

\textbf{Acknowledgments} We are greatly indebted to Guillaume Bal for kindly pointing out that the operator $\mathcal{L}$ is not elliptic, as we had assumed in our earlier draft, and generously helping us correct the error. We would also like to thank Daniel Sharp for providing us with an mParT implementation, which we could readily use to generate Figure 4.

\bibliography{main.bib}

\appendix
\section{Other iterations based on the contraction mapping principle}
\label{appx:fixedPoint}
For a fixed $p$, the score operator, $\mathcal{G}(p,\cdot)$, is a map from the space of transformations to scores. We first define a self-map corresponding to the score-residual operator. Then, we define a fixed point iteration for the self-map that converges in a setting where it is a contraction.

For convenience, we denote by $\mathcal{G}_p$ the operator $\mathcal{G}(p,\cdot)$ that returns the transported score, given a transport map. Let $\mathbb{T}$ be a Banach space of functions on $\mathbb{R}^d$ and $\mathbb
{S}$ a Banach space of scores. Using the definition of the score operator,  
\begin{align}
    \mathcal{G}_p(T) \circ T = p (\nabla T)^{-1} - {\rm tr}((\nabla T)^{-1}\: \nabla^2 T) \: (\nabla T)^{-1} = q\circ T,
\end{align}
when $T$ is a solution of the score-matching problem. We assume that the target score $q$ is a homeomorphism onto its image. This is indeed a strong condition that multi-modal distributions, for example, do not satisfy.
Define a self-map $\mathcal{H}:\mathbb{T}\to\mathbb{T}$   
\begin{align}
    \mathcal{H}(T) = q^{-1} \circ  \mathcal{G}_p(T)\circ T.
\end{align}
Clearly, if $T$ is a solution of the score-matching problem, it is a fixed point of $\mathcal{H}$. Near a fixed point, we can define the following fixed-point iteration of $\mathcal{H}$.
\begin{align}
\label{eq:fixedPointIteration}
    T_{n+1} = \mathcal{H}(T_n) = q^{-1} \circ \mathcal{G}_p(T_n) \circ T_n,
\end{align}
with an arbitrary map $T_0$. When $\mathcal{H}$ is a contraction near its fixed point, say, $T^*$, then $T_n$ defined in \eqref{eq:fixedPointIteration} converges to $T^*$, by the contraction mapping principle. Consider one sufficient condition: when the functional derivative  $D(\mathcal{G}_p(T)\circ T)(T^*)$ and the gradient $\nabla q^{-1}$ are small in the operator norms and $C^0$ norm respectively. Then, one obtains that $\mathcal{H}$ is contraction, and hence the fixed-point iteration defined above converges to a fixed point or a transport map.

However, this is only a sufficient condition for convergence. When $\mathcal{H}$ is not contractive, the fixed point iterations may not converge, or may exhibit slower than exponential convergence. When $\mathbb{T}$ is large enough to contain multiple fixed points of $\mathcal{H}$, the iterations may oscillate between basins of attraction of multiple fixed points. 
Note that the sufficient conditions for $\mathcal{H}$ to be a contraction impose restrictions on the
behavior of the score operator near the fixed point and on the second-order derivatives of the target, which may preclude multi-modality in the target. The SCONE method, on the other hand, does not explicitly impose such a restriction at the fixed point and is hence more general.

 \section{SCONE example}
 \label{appx:gaussian}
As the first example, suppose the target is a univariate Gaussian with mean $m$ and variance $s^2$, while the source distribution is a standard normal in 1D. In this case, $q(x) = -(x - m)/s^2$ and $p(x) = -x$. The SCONE update $v_n$ satisfies the following ODE:
\begin{align}
\label{eq:newtonUpdate1D}
    p_n - q = v_n'' + q v_n' + q' v_n.
\end{align}
At $n = 0$, $p_n = p$. The above ODE is defined on an unbounded domain, without specific decay rates for the solution at the boundaries, $\pm \infty$. This allows for unbounded solutions. Notice that the solution is always affine since, at $n=0$, the left-hand side is affine. Subsequently, the update equation is satisfied by affine functions $v_n$ at every $n$ and hence $T_n$, which is a composition of affine functions, is affine. By comparing coefficients to solve for the update equations, we can obtain recurrence relationships for the slopes and intercepts of $v_n$, $T_{n+1} \coloneqq ({\rm Id} + v_n)\circ T_n$, and $p_n = (p/T_n')\circ T_n^{-1}$. We can inductively show that all three are affine functions for all $n$.
In particular, if $p_n(x) = a_n x + b_n$ and $v_n(x) = (A_n - 1)x + B_n$, we obtain,
\begin{align*}
    A_n &= -a_n s^2/2 + 1/2\\
    B_{n} &= -b_n s^2 + m/2 - a_n m s^2/2,
\end{align*}
by comparing coefficients in the update \eqref{eq:newtonUpdate1D}. Then, the update for the score gives, 
\begin{align*}
    a_{n+1} = a_n/A_n^2, b_{n+1} = -a_n B_n/A_n^2 + B_n/A_n^2.
\end{align*}
Considering the set of sequences $\{a_n, b_n, A_n, B_n\}_n$, it is clear from the relationships above that when $A_n \to 1$ and $B_n \to 0$, $a_n \to (-1/s^2)$ and $b_n \to  m/s^2$. Thus, when the iterations for $T_n$ converge, or equivalently, when $v_n \to 0$, $p_n \to q$. Moreover, in this case, the limit $T$ is the function $T(x) = sx + m$, which coincides with the increasing rearrangement on $\mathbb{R}$ (and hence the optimal map). The intermediate distributions corresponding to the scores $p_n$ are all Gaussian. 

Notice that since this convergence can be established for all $s$ and $m$, it suggests that SCONE transport converges even when in H\"older norm, $\|p-q\|$ is not small. That is, even though the derivation of the is premised on the local expansion of the score operator around $(q, {\rm Id})$, the smallness of $p-q$ and $T-{\rm Id}$ is not a necessary condition for the convergence of the method. 
\begin{figure}
\centering
\includegraphics[width=0.31\textwidth]{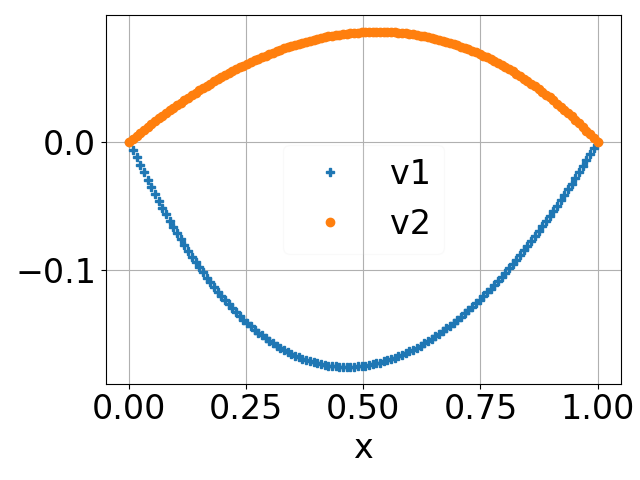}
\includegraphics[width=0.31\textwidth]{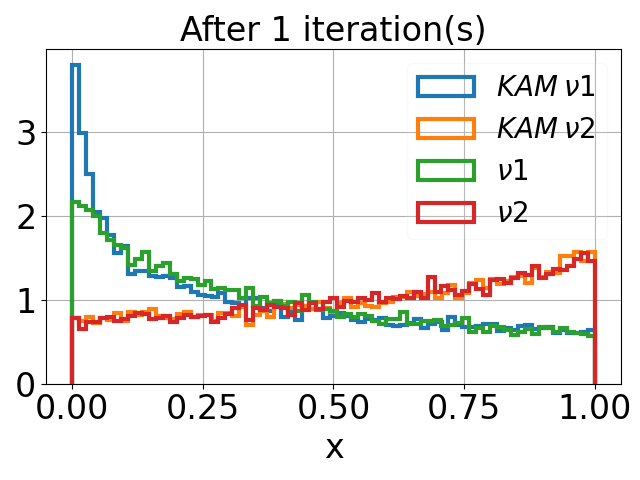}
\includegraphics[width=0.31\textwidth]{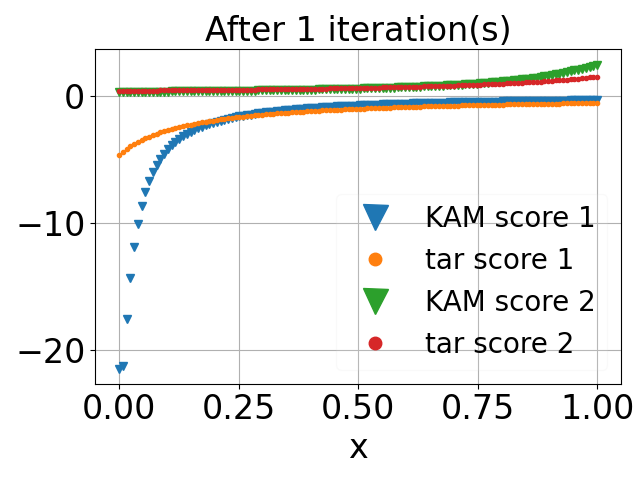} \\
\centering
\includegraphics[width=0.31\textwidth]{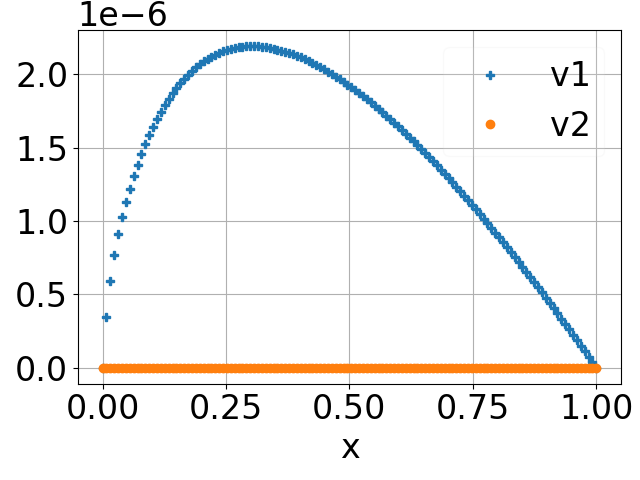}
\includegraphics[width=0.31\textwidth]{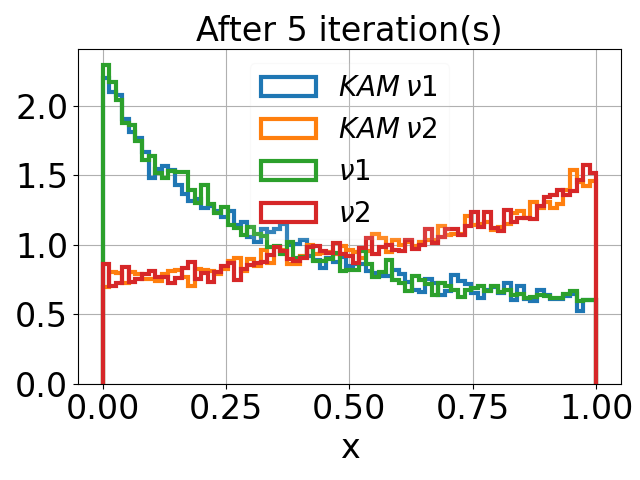}
\includegraphics[width=0.31\textwidth]{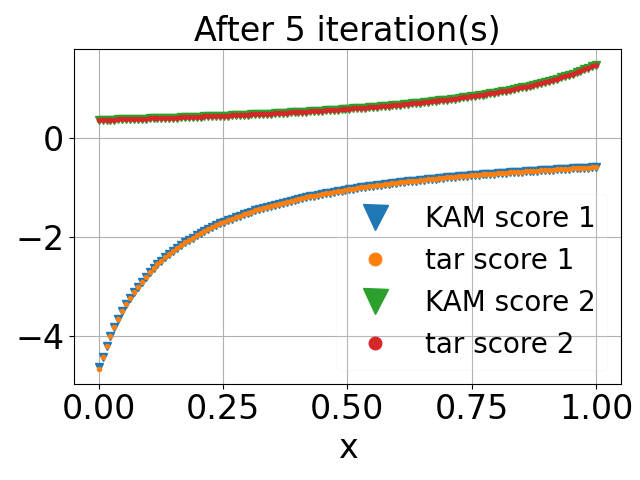} 
    \caption{Numerical validation of SCONE transport on 1D densities: the first row depicts results after 1 iteration of our Newton method and the second row after 5 iterations. The scalar field $v,$ the histogram approximations of the target density and the target score are plotted for the two different targets described in section \ref{sec:appx-num}. The results of the computed transformed densities and scores from the SCONE iteration are compared against the target densities and scores in columns 2 and 3.}
    \label{fig:kam1d}
\end{figure}
\begin{figure}
    \centering
\includegraphics[width=0.3\textwidth]{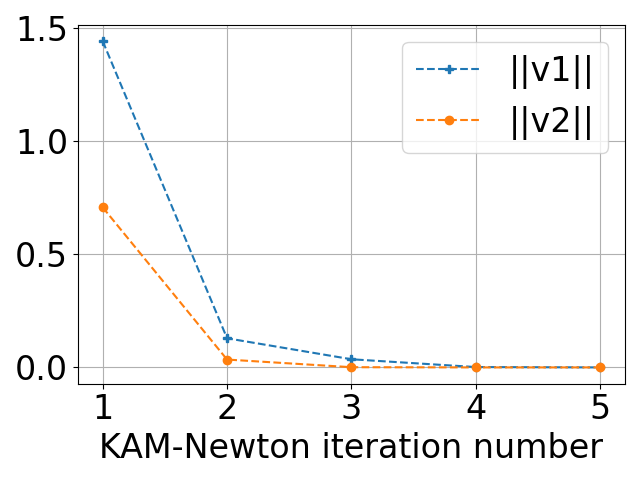}
\includegraphics[width=0.3\textwidth]{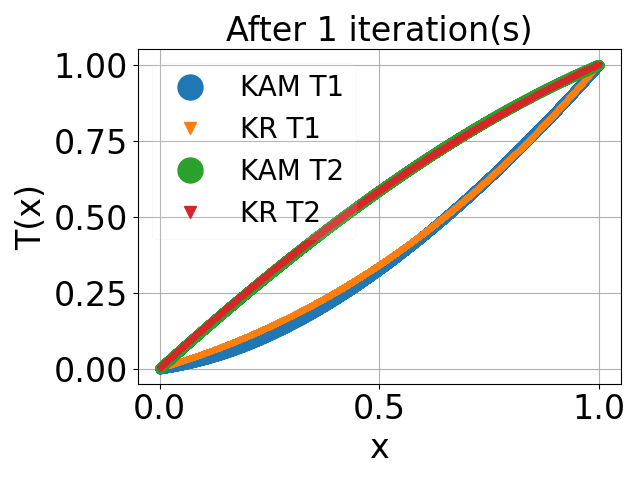}
\includegraphics[width=0.3\textwidth]{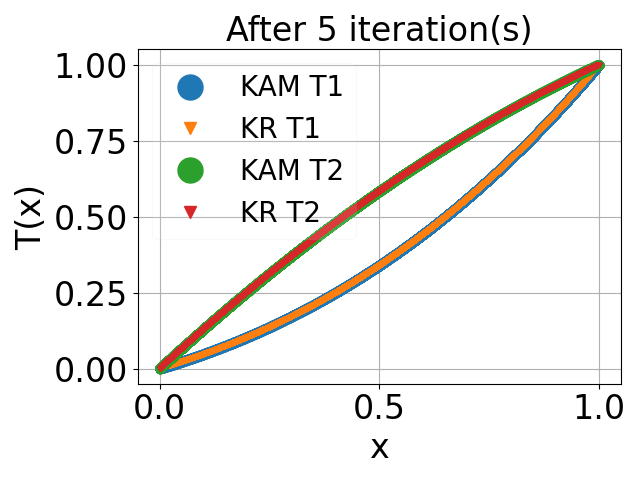}
    \caption{On the left column, we show the convergence of our algorithm ($\|v_n\| \to 0$) for the two different target densities given in section \ref{sec:appx-num}. The center and right hand side figures compare the optimal map computed analytically against the SCONE transport map computed numerically after one and five iterations respectively.}
    \label{fig:convergence}
\end{figure}

\section{Additional experiments}
 \label{sec:appx-num}
Considering 1D targets, we first give proof-of-concept numerical results to validate our SCONE construction. We consider two smooth densities supported on $[0,1]$ as our targets: $\rho_{\nu1} = ((x+1)^3 - 1)/7$ and $\rho_{\nu2} = 4/3 - (2-x)^2/3.$ These densities are shown as histograms in the second column of Figure \ref{fig:kam1d}. In 1D, the PDE that needs to be solved at every SCONE iteration is an ODE, which we solve using a finite difference method with 128 grid points and zero Dirichlet boundary conditions on $[0,1].$ Our source density is the uniform (Lebesgue) density on $[0,1]$ whose score is the zero function. As shown in Figure \ref{fig:kam1d} (third column), the transformed score matches the target score within a few SCONE iterations. The solution $v$ also quickly approaches the zero function as confirmed in the first column of Figure \ref{fig:kam1d} and Figure \ref{fig:convergence}(left), thus establishing numerically the convergence of our construction (see Theorem 3 in the main text). 
On one-dimensional domains, the optimal transport map (for a variety of costs) is simply the increasing rearrangement (see \citep{santambrogio2015optimal}, Chapter 2), which can be computed analytically in our setting, and is shown in the second and third columns of Figure \ref{fig:convergence}. 
As shown, the SCONE construction converges to this  transport map.
\subsection{1D unbounded target: bimodal Gaussian with equally weighted modes}
\label{sec:equalWeights}
Next, we consider a one-dimensional bimodal Gaussian target (shown in Figure \ref{fig:bimodal0.5}), $0.5 \mathcal{N}(-2,1) + 0.5 \mathcal{N}(2,1)$. We again solve the ODE with finite difference on $[-10,10]$. The source is taken to be the standard Gaussian (unimodal distribution). In Figure \ref{fig:bimodal0.5}, the first row presents results obtained after 1 iteration of our SCONE algorithm; the second row, after 3 iterations. We see that, after just 1 iteration, the two modes are detected, although the density and scores are not well-approximated. After 3 iterations, the empirical density (shown in blue) of the samples transported by the SCONE transport map match the target density (orange line plot) closely. As shown in Figure \ref{fig:bimodal0.5} (second row), as the SCONE update solution $v$ declines, the density and the scores approach their target values. In Figure \ref{fig:bimodal0.5_conv} (left), we show the convergence of the solution $v$. From the figure, the convergence appears to be exponentially fast, with the rate decreasing after the first 4 iterations. The final transport map (right) is taken after 5 iterations, which accurately models the target density (center). A grid size of 4096 is used for solving the SCONE iteration, but grid size reductions ${\cal O}(1000)$ produces similar convergence results. We find that more iterations are needed when the modes are well-separated, e.g., sampling from an equally weighted bimodal distribution with modes centered at -4 and 2 required around 20 iterations. 
\begin{figure}
	\includegraphics[width=0.3\textwidth]{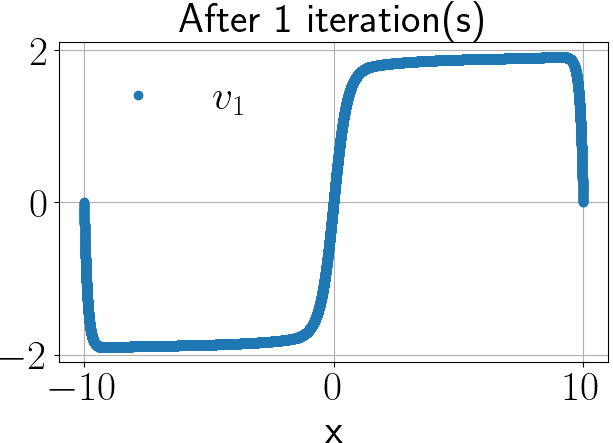}
	\includegraphics[width=0.32\textwidth]{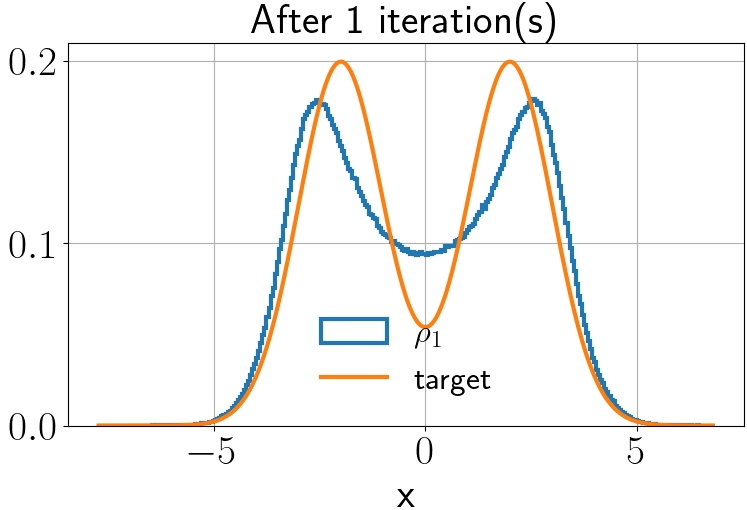}
	\includegraphics[width=0.32\textwidth]{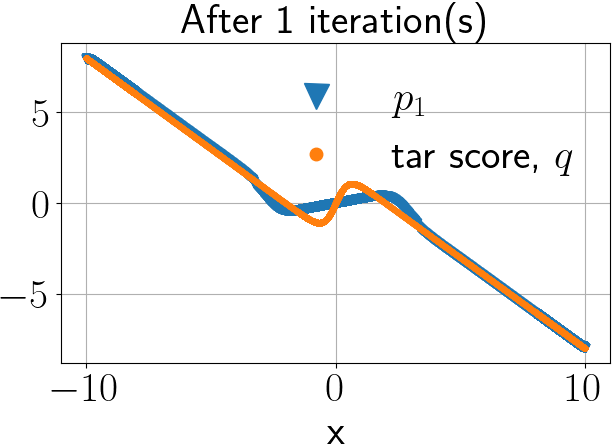}\\
	\includegraphics[width=0.3\textwidth]{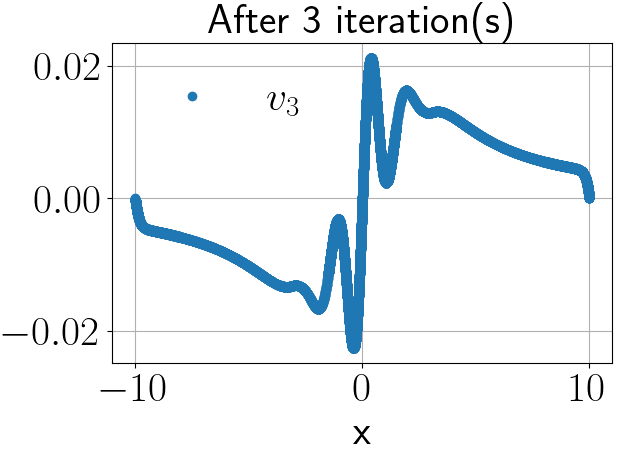}
	\includegraphics[width=0.32\textwidth]{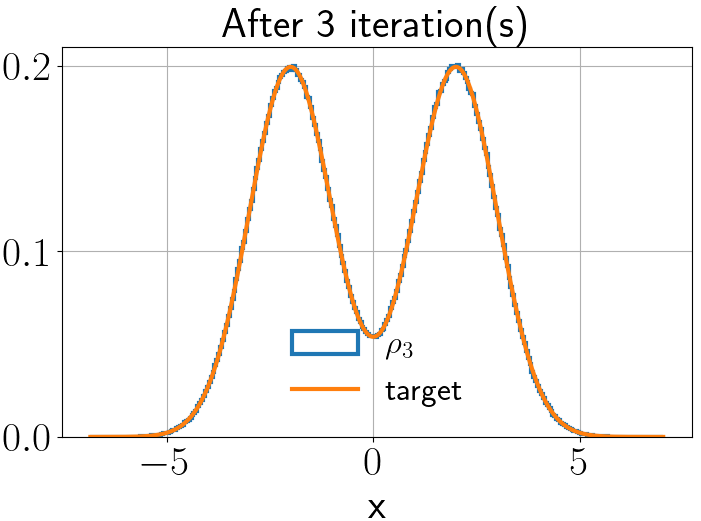}
	\includegraphics[width=0.32\textwidth]{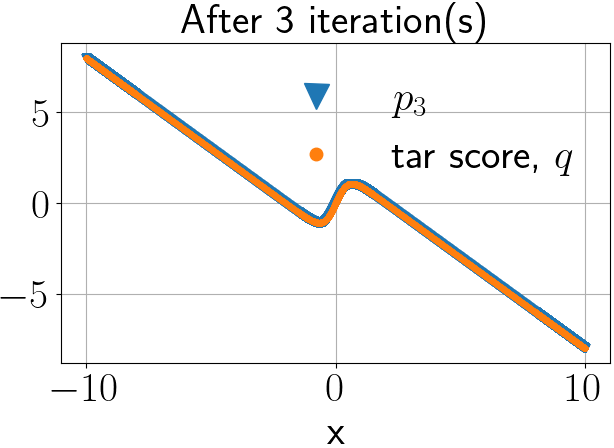}
	\caption{SCONE algorithm results: the target distribution, shown in orange in the center column, is an equally weighted bimodal Gaussian. The application of the SCONE algorithm, as described in section \ref{sec:equalWeights}, is shown in the first row after 1 iteration and in the second row after 3 iterations. The scalar field $v,$ the histogram approximations of the target density and the target score are plotted in the first, second and third columns, respectively. The results of the computed transformed densities and scores from the SCONE iteration are compared against the target densities and scores in columns 2 and 3.}
	\label{fig:bimodal0.5}
\end{figure}
\begin{figure}
	\includegraphics[width=0.3\textwidth]{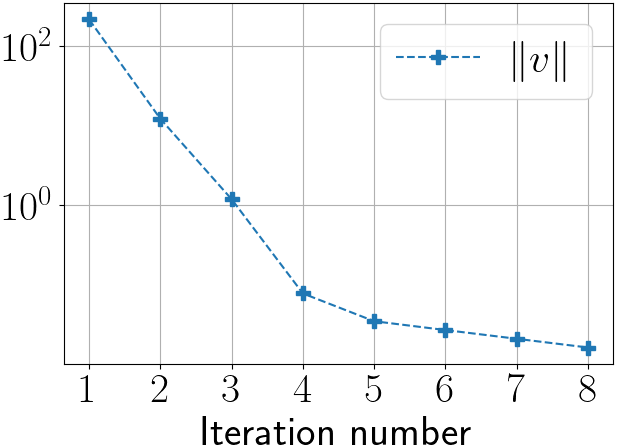}
	\includegraphics[width=0.32\textwidth]{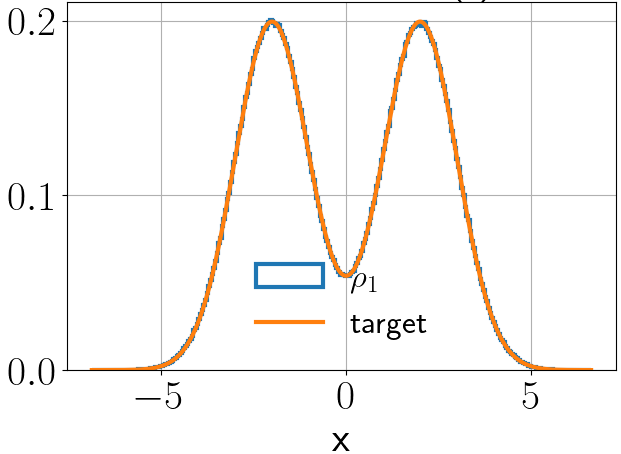}
	\includegraphics[width=0.32\textwidth]{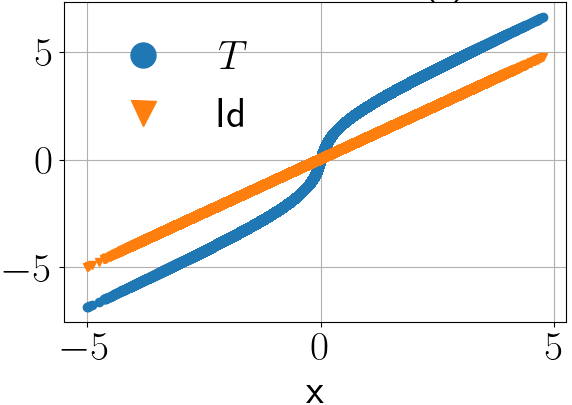}
	\caption{Convergence of SCONE iteration: the target distribution, shown in orange in the center column, is an equally weighted bimodal Gaussian. The application of the SCONE algorithm, as described in section \ref{sec:equalWeights}, results in the convergence of $\|v_n\|$, shown in the first column. The final SCONE transport map after 5 iterations is on the right (blue).}
	\label{fig:bimodal0.5_conv}
\end{figure}
\subsection{1D unbounded target: bimodal Gaussian with unequally weighted modes}
Previously in section \ref{sec:equalWeights}, when the target consisted of equally weighted modes, we used a finite difference method to solve the ODE, and then function interpolation to evaluate $v$ at the samples. Using these interpolated values, the score function, $p$ was updated, and the iterations continued by solving the ODE again. This vanilla scheme leads to numerical blow-up when the target has unequal weights. The reason is that errors in the solution of $v$ leads to the divergence of our Newton-like method. We observe that adding a small regularization term in the finite-difference ODE solution ($\ell^2$ regularization parameter set to 0.01), along with refining the grid near the points where $q'$ is large, induces convergence. The results are in the main text (section 5).

We provide the source code implementing the SCONE algorithm on all the examples above in \citep{code}. This contains `oneD.py' that implements the SCONE algorithm on all the examples above. The source file `oneD\_nonUni.py' implements adaptive grid refinement in the finite-difference solver. The unit tests that generate all the figures in this section are in `tests/test\_1D.py'. The code is written in Python and uses numpy and scipy.

\end{document}